\documentclass[journal]{IEEEtran}
\usepackage{paralist,tikz,graphicx,amsmath, amsfonts,amssymb,amsthm,overpic, latexsym,epsfig,color,rotating,paralist,times,float,subcaption,algorithm,verbatim,enumitem,url}
\usetikzlibrary{fit,positioning}
\usepackage{tikz,pgfplots,stfloats}

\usepackage[font=scriptsize]{caption}

\usepackage{tikz}
\usetikzlibrary{patterns}

\usepackage{amsmath}
\usepackage{amssymb}
\usepackage{amsthm}
\usepackage{bbm}
\usepackage{color}
\usepackage{algorithm}
\usepackage{algpseudocode}
\usepackage{url}
\usepackage{mathtools}
\usepackage{xfrac}

\usepackage{enumitem}
\usepackage{geometry}
\usepackage{tikz}
\usetikzlibrary{calc,patterns,decorations.pathmorphing,decorations.markings,arrows}
\usetikzlibrary{quotes,arrows.meta}

\tikzstyle{block} = [draw, rectangle, minimum height=3em, minimum width=3em]
\tikzstyle{sum} = [draw, circle, node distance=1.5cm]
\tikzstyle{input} = [coordinate]
\tikzstyle{output} = [coordinate]

\newcommand{\skor}{\mathcal{S}}

\newcommand{\ones}{\mathbf{1}}

\newcommand{\E}{\mathbb{E}}
\newcommand{\obs}{Y}

\newcommand{\reals}{\mathbb{R}}

\newcommand{\ind}{\boldsymbol{1}}
\newcommand{\CE}{\mathbb{E}}

\renewcommand{\leq}{\leqslant}
\renewcommand{\geq}{\geqslant}
\renewcommand{\epsilon}{\varepsilon}
\renewcommand{\phi}{\varphi}
\renewcommand{\th}{\theta}
\newcommand{\noise}{z}
\newcommand{\wloss}{C}
\newcommand{\defn}{\overset{\textnormal{defn}}{=}}

\newcommand{\argmin}{\operatorname{argmin}}
\newcommand{\loss}{\ell}
\newcommand{\Xt}{X^\th}
\newcommand{\Loss}{L}
\newcommand{\kernel}{K}
\newcommand{\Th}{\Theta}
\newcommand{\CF}{\mathcal{F}}
\newcommand{\1}{\boldsymbol{1}}

\newtheorem{theorem}{Theorem}

\title{\vspace{-1.3cm} Malliavin Calculus with Weak Derivatives for Counterfactual Stochastic Optimization}
\author{Vikram Krishnamurthy, \and Luke Snow \thanks{Vikram Krishnamurthy and Luke Snow are  with the  School of Electrical and Computer Engineering, Cornell University.
    Email: vikramk@cornell.edu  and las474@cornell.edu. \\  This  research was supported by the 
    National Science Foundation under grants CCF-2312198 and  CCF-2112457.}
}
\date{}

\begin{document}
\maketitle
\thispagestyle{empty}


\begin{abstract}

  We study counterfactual  stochastic optimization  of  conditional loss functionals under   misspecified and noisy gradient information. The difficulty is that when the conditioning event  has vanishing or zero probability, naïve Monte Carlo estimators are prohibitively inefficient; kernel smoothing, though common, suffers from slow convergence. We propose  a two-stage kernel-free methodology. First, we show using Malliavin calculus that the conditional loss functional of a diffusion process admits an exact
  representation  as a Skorohod integral, yielding variance comparable to  classical Monte-Carlo variance.
  Second, we establish that a weak derivative estimate of the conditional loss functional with respect to
  model parameters can be evaluated with constant variance, in contrast to the widely used score function method whose variance grows linearly in the sample path length. Together, these results yield an efficient framework for counterfactual conditional stochastic gradient algorithms in rare-event regimes.
    \end{abstract}

\section{Introduction and Problem Formulation}

Consider the  stochastic differential equation (SDE)
\begin{equation}
  \label{eq:sde}
  dX_t = b_\th(X_t,t) dt + \sigma(X_t,t) dW_t, \quad t \in [0,T]
\end{equation}
where $W$ denotes $d$-dimensional standard Brownian motion.
Our  aim is estimate  the minimizer of the conditional loss function 
\begin{equation}
  \label{eq:objective}
  \argmin_{\th \in \Th} \Loss(\th) =  \E[  \ell(X^\th)   \mid  g(X^\th) = 0]
\end{equation}
where $\Th$ is a compact subset of $\reals^p$, $\Loss(\cdot)$ is continuous, and $\loss(\cdot), g(\cdot)$ are functionals. In addition, we assume that: \\  (i) The functions  $(b_\th,\sigma,\ell,g)$ are known.  \\
       (ii) We are given $N$ simulated sample paths of  $X$, but we cannot control these sample paths to ensure $g(\Xt) = 0$. \\
       (iii) The event  $\{g(\Xt) = 0\}$  has low (zero) probability.

For the purpose of exposition, we can re-express the loss as
\begin{equation}
  \label{eq:equiv_loss}
  \Loss(\th) = \frac{\E\{ \ell(X^\th) \,\delta(g(X^\th)) \}}{\E\{\delta(g(X^\th))\}   }
\end{equation}
where $\delta(g(X))$ denotes the Dirac delta  centered at zero.

{\bf Example}. To illustrate the main idea, suppose we choose
$$\ell(X) = \int_0^T h(X_s) ds , \quad g(X) = X_{T/2}-x, \quad x \in \reals^n.$$  Here the loss
functional $\ell$ is specified by a suitably chosen 
 function $h(\cdot)$.
 Also $g(X)=0$ imposes an anticipatory constraint on the sample path at time $T/2$.  The counterfactual optimization asks: {\em Given sample paths of $X$ that we cannot control or simulate directly,  how can  we  minimize  $\Loss(\th)$ under the hypothetical condition that  the sample  paths pass through  specific point $x$ at time $t=T/2$?} 
 Even with full  control over simulations of $X$, it is infeasible to generate sample paths that satisfy the zero-probability event  $\{X_{T/2}=x\}$, for  fixed $x$.  More generally, for functional constraints such as 
 $g(X)  = \int_0^T \gamma(X_s) ds$,  no feasible simulation strategy can directly enforce $g(X)=0$.


 {\bf Limitation of Kernel Methods}. 
Naive Monte-Carlo  estimation of  $\Loss(\th)$ in~\eqref{eq:equiv_loss} fails   due to Dirac delta in  the denominator. 
The classical workaround is to use a kernel method: approximate  the Dirac delta  $\delta(g(X))$  by a kernel $\kernel_\Delta(g(X))$ where $\Delta$ denotes the kernel bandwidth. Typically  $\kernel_\Delta$ is  a multivariate Gaussian density and $\Delta$ controls its variance.
The kernel-based Monte-Carlo estimator for the loss $\Loss$ given $N$ independent realizations $X^{(i)},i=1,\ldots,N$ of $X$ is
$$ \hat{\Loss}(\th) = \frac{ \sum_{i=1}^N \ell(\th,X_{[0,T]}^{(i)}) \, \kernel_\Delta(g(X^{(i)})) }{\sum_{i=1}^N  \kernel_\Delta(g(X^{(i)}) }. $$
But the variance of the estimate of $\Loss(\th)$ depends on the kernel bandwidth $\Delta$ and convergence becomes  excruciatingly slow for large~$n$ or small-probability events $\{g(\Xt)=0\}
$.

\subsection{Main Results} This paper  develops  a two-stage kernel free approach for counterfactual stochastic optimization: \\ 
(i) \textit{Loss evaluation via Malliavin calculus}. We show that  
 $\loss(\Xt) \delta(g(\Xt))$ and $\delta(g(\Xt))$ in~\eqref{eq:equiv_loss} admit exact Skorohod integral representations. Their  expectation can therefore be computed using classical Monte-Carlo.
For $N$ independent trajectories generated by~\eqref{eq:sde}, the estimator
achieves $O(1/N)$ variance, identical to  classical Monte-Carlo, even in rare-event settings \cite{FLL01,BET04,Cri10}.
\\
 (ii) \textit{Gradient estimation via weak derivatives}. We show that the gradient $\nabla_\th \Loss(\th)$ can be estimated efficiently using a weak derivative approach based on the Hahn-Jordan decomposition. The variance of the gradient estimate is $O(1)$. This is in comparison to the widely used score function estimator which has variance $O(T)$. Weak derivative estimators are studied extensively in \cite{Pfl96,HV08,KV12}.

 By combining (i) and (ii), we obtain a counterfactual stochastic gradient algorithm that converges to a local stationary point of~$\Loss(\th)$.
The procedure is displayed in Figure~\ref{fig:stochgrad}.

\begin{figure}[h]
  
  \begin{tikzpicture}[scale=0.7, transform shape,node distance = 4.5cm, auto]
    \tikzstyle{arrow} = [thick,->,>=stealth]
    \tikzset{
    block/.style={rectangle, draw, line width=0.5mm, black, text width=4em, text centered,
                 minimum height=2em},
               line/.style={draw, -latex}}
   \tikzset{
    block2/.style={rectangle, draw, line width=0.5mm, black, text width=6em, text centered,
                 minimum height=2em},
               line/.style={draw, -latex}}             

             \node[block](Markov){SDE};
  \node[block2,right of = Markov,node distance=3.2cm](sensor1){Malliavin for loss evaluation};             
  \node[block2,right of = sensor1,node distance=3.2cm](sensor){Weak Derivative};
  \node[block2,right of=sensor, node distance=4.2cm](filter){Stochastic Gradient Algorithm};
  \node[right of=filter,node distance=2.2cm](nullnode){};
  \draw[-Latex](Markov) -- node[above] {$X^{\th_n}$}  (sensor1);
  \draw[-Latex](sensor) -- node[above] {$\hat\nabla_\th \ell(X^{\th_n})$} (filter);
 \draw[-Latex](sensor1) --  (sensor);
  
  \draw[-Latex](filter) -- node[above,pos=0.55] {$\th_{n+1}$} (nullnode);
  \node[draw,inner sep=4pt,dashed,fit={(sensor1) (filter)},label={Counterfactual Stochastic Approximation   Algorithm}] {};
  \draw[-Latex] (filter.east) -- ++(1,0) |- ([yshift=-0.5cm]Markov.south)
  -|
  (Markov.south);
\end{tikzpicture}
\caption{Counterfactual Stochastic Gradient Algorithm}
\label{fig:stochgrad}
\end{figure}
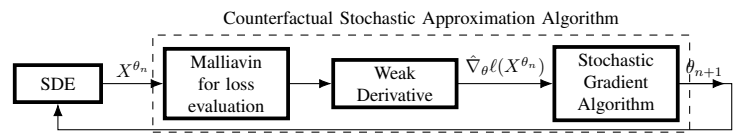

{\em Remark: Model fitting}.
 The above framework 
 aims to choose $\th$  to control  SDE~\eqref{eq:sde} to minimize the conditional loss $\Loss(\th)$. The framework also applies to model fitting: fit 
 the SDE~\eqref{eq:sde} to $N$ externally  generated data trajectories  $\obs_{0:T}^{(1)},\ldots \obs_{0:T}^{(N)}$. In this case, one seeks to minimize the conditional loss  $L(\th) = \E\{\loss(Y,\Xt) | g(\Xt) = 0 \}$.  

 \subsection{Context. Passive Learning}
 The above counterfactual stochastic optimization framework also arises in 
 passive learning,  goal conditioned diffusion models, and stochastic optimization with safety/anticipatory constraints.
 To give additional insight, we briefly discuss our problem in terms of passive stochastic approximation.
 
 In classical stochastic approximation we  observe a sequence of noisy gradients $\{\nabla \loss(\th_k,\noise_k)\}$ where $\noise_k$ is a noisy signal and  $\nabla \loss(\th,\noise) $ is an asymptotically unbiased estimate of $\nabla \Loss(\th)$. We optimize
$\Loss(\th) = \E\{\ell(\th_k,\noise_k)\}$ via the stochastic gradient algorithm  
$$ \th_{k+1} = \th_k - \epsilon \nabla \ell(\th_k,\noise_k) $$
Under reasonable conditions \cite{KY03}, the interpolated trajectory of the estimate $\{\th_k\}$ converges weakly to the ordinary differential equation (ODE)
\begin{equation}
  \label{eq:ODE}
 \frac{d\th}{dt} = \nabla L(\th) . 
\end{equation}    

 {\bf Passive Learning}.  
    In passive stochastic optimization \cite{NPT89,YY96,KY22}, unlike classical stochastic gradient, we observe a sequence of noisy and {\em misspecified gradients}: $\{\alpha_k, \nabla \loss(\alpha_k,\noise_k)\}$, where the parameters $\alpha_k\in \reals^p$ are chosen randomly   according to  probability density $p(\cdot)$, potentially by an adversary.  The passive stochastic gradient algorithm is 
    \begin{equation}
      \label{eq:passive}
    \th_{k+1} = \th_k - \epsilon\,\kernel_\Delta(\alpha_k-\th_k )\,\nabla \ell(\alpha_k,\noise_k) . 
  \end{equation}
  The kernel $\kernel(\cdot)$ weights the usefulness of the gradient $\nabla\loss(\alpha_k,\noise_k)$ compared to the required gradient
$\nabla\loss(\th_k,\noise_k)$.
If $\th_k$ and $\alpha_k$ are far apart, then kernel is   small and  only a small proportion of the gradient estimate $\nabla\loss(\alpha_k,\noise_k)$ is added to the stochastic gradient algorithm. On the other hand, if $
\alpha_k = \th_k$,
the algorithm  becomes the  classical  stochastic gradient algorithm. 
Under reasonable conditions, for small bandwidth parameter $\Delta$, the kernel $\kernel_\Delta$ behaves as a Dirac delta and   the interpolated trajectory converges weakly to the ODE
$$\frac{d \th}{dt} = \int_\Theta \pi(\theta) \,\delta(\alpha- \th)\, \nabla L(\alpha)\, d\alpha  =p(\th) \, \nabla L(\th).
$$
Notice that this ODE has the same fixed points as~\eqref{eq:ODE}.

{\bf Counterfactual Learning}.
Finally, the counterfactual stochastic optimization problem described above, can be regarded as a passive stochastic optimization problem.  At each stage, we require gradient estimates  
$\nabla_\th \loss(\th,g(\Xt)=0,\noise) $ that are unbiased estimates of
$\nabla_\th \Loss(\th) $ where $\Loss(\th)$ is defined in~\eqref{eq:objective}, but we are instead provided with  noisy and misspecified gradient estimates 
$\nabla_\th \loss(\th,g(\Xt)=a,\noise) $ for random  $a \in \reals$. That is, while
the desired gradient corresponds to the counterfactual constraint
  $[\th,g(\Xt)=0]$,  we only observe  the misspecified noisy gradient evaluated at
$\alpha = [\th,g(\Xt)=a]$.  Therefore, one can use the passive kernel based algorithm~\eqref{eq:passive} to solve the counterfactual stochastic optimization problem.  However, in this paper, we will exploit the structure of the SDE~\eqref{eq:sde} and not use the kernel based algorithm.

\section{Malliavin Calculus Approach to Estimate Conditional Loss}

Malliavin calculus \cite{Nua06} was  developed in the 1970s as a probabilistic method to prove H\"ormander hypoellipticity theorem for the solution of SDEs.  It was later adapted in mathematical finance to compute sensitivities (Greeks) of option prices.  Here, as in \cite{FLL01,BET04,Cri10}, we employ  Malliavin calculus  to \textit{efficiently} evaluate the conditional loss $\E\{\loss(\Xt)|g(\Xt)=0\}$, even when the conditioning event $\{g(\Xt)=0\}$ has vanishing or zero probability. 


\subsection{Preliminaries. Malliavin Calculus.}
\label{sec:mall_pre}

We briefly recall the two central objects.

\subsubsection{Malliavin derivative}
 We work on the probability space $(\Omega,\mathcal{F},\mathbb{P})$ with  a $d$-dimensional Brownian motion $W = (W^1,\dots,W^d)$ and the natural filtration $\{\CF_t\}_{t\geq0}$. 
For a smooth functional $F$ of $W$, the \emph{Malliavin derivative} $D_t F$ is defined as the process measuring 
the infinitesimal sensitivity of $F$ to perturbations of the Brownian path at time $t$. Formally, 
for cylindrical random variables of the form
\[
F = f\bigg( \int_0^T h_1(s)\, dW_s, \dots, \int_0^T h_n(s)\, dW_s \bigg),
\]
with $f \in C_b^\infty(\mathbb{R}^n)$ and $h_i \in L^2([0,T];\mathbb{R}^d)$, the derivative is
\[
D_t F = \sum_{i=1}^n \frac{\partial f}{\partial x_i}\bigg( \int_0^T h_1\, dW, \dots, \int_0^T h_n\, dW \bigg)\, h_i(t).
\]
The closure of this operator in $L^p$ leads to the Sobolev space $\mathbb{D}^{1,p}$ of Malliavin differentiable random variables.

\paragraph{Skorohod integral}
The adjoint of the Malliavin derivative is the \emph{Skorohod integral}, denoted $\skor(u)$.  Indeed, for a process $u \in L^2([0,T]\times \Omega;\mathbb{R}^d)$, $u$ is in the domain of $\skor$ if there exists a square-integrable 
random variable $\skor(u)$ such that for all $F \in \mathbb{D}^{1,2}$,
\begin{equation}
\label{eq:adjoint}
\E[F \, \skor(u)] = \E\!\left[\int_0^T \langle D_t F, u_t \rangle \, dt \right].
\end{equation}
The above adjoint relationship  serves as the definition of the Skorohod integral an can be written abstractly as  $$\langle F, \mathcal{S}(u) \rangle_{L^2(\Omega)} =
\langle D F, u \rangle_{L^2([0,T] \times \Omega)}.$$

When $u$ is adapted to the filtration $\{\CF_t\}_{t\geq0}$, the Skorohod integral $\skor(u)$ coincides with the Itô integral $\int_0^T u_t\, dW_t$. 
In general, $\skor(u)$ extends stochastic integration to non-adapted processes and is sometimes called the \emph{divergence operator}.

\subsubsection{Integration by parts}
The duality relation~\eqref{eq:adjoint}  yields the Malliavin’s integration-by-parts formula, which underpins many applications, 
including Monte Carlo estimation of conditional expectations and sensitivity analysis for SDEs 
(see \cite{Nua06,FLL01,BET04}). 

\subsubsection{Computing Malliavin Derivative and Skorohod Integral} 
\label{sec:mall_comp}
The following properties are the key tools which allow us to compute the Malliavin derivative and Skorohod integral: 
\begin{enumerate}
    \item \textit{Malliavin derivative of diffusion}. For  diffusion process $\{X_t\}_{t\geq 0}$ \eqref{eq:sde}, the  Malliavin derivative $D_sX_t$ is \cite{GM05}
    \begin{equation}
    \label{eq:mall_form}
        D_sX_t = Y_tZ_s\sigma(X_s,s)\1_{s\leq t}
    \end{equation}
    where  $Y_t:= \nabla_x X_t$ is the Jacobian matrix and $Z_t$ is its inverse $Z_t := Y_t^{-1}$.  This, together  with the Malliavin chain rule \cite{Nua06}, facilitates evaluating  Malliavin derivatives of general functions of diffusions.
    \item \textit{Skorohod expansion}. For random variable $F\in \mathbb{D}^{1,2}$ and Skorohod-integrable process $u$, we have \cite[eq. 2.2]{GM05}:
    \begin{equation}
    \label{eq:skor_exp}
    \skor(F u) = F\skor(u) - \int_0^TD_tF\cdot u_t dt
    \end{equation}
    In general,  the Skorohod integrand $\{u_t\}_{t\in[0,T]}$ of interest may be  non-adapted. However, in the special case where $u$ factorizes into  the product of an adapted process $\hat{u}=\{\hat{u}_t\}_{t\in[0,T]}$ and an anticipative random variable $F$, this formula gives a constructive expression.  Specifically, we can expand $\skor(u) = \skor(F\hat{u}) $ using  \eqref{eq:skor_exp} and compute it in terms of a  standard It\'o integral of the adapted part $\hat{u}$ together with the  Malliavin derivative of the anticipatory random variable $F$. 
    
\end{enumerate}

\subsection{Malliavin Calculus Expression for Conditional Expectation}
The following main result expresses   the conditional expectation \eqref{eq:objective} as the ratio of unconditional expectations. 
\begin{theorem}  Assume
$\ell(\Xt), g(\Xt) \in L^2(\Omega)$ and  $D_t\ell(\Xt), D_tg(\Xt) \in L^2(\Omega \times [0,T])$. Then  the conditional loss $\Loss$ in~\eqref{eq:equiv_loss} is
\begin{align}
\begin{split}
\label{eq:malliavin}
 &\Loss(\th)  =\E[ \ell(X^\th)  \mid g(X^\th) = 0]  =  \frac{E_1^\th}{E_2^\th} \\
   &\text{ where } \\&E_1^\th 
 = \E\bigg[ \ones_{\{g(\Xt)>0\}}\left(\ell(X^\th) \skor(u)  - 
  \int_0^T (D_t \ell(\Xt)) u_t dt\right) \bigg]\\  &E_2^\th =
\E[  \ones_{\{g(\Xt)>0\}} \skor(u)]
\end{split}
\end{align}
Here 
$u$ is any process that satisfies  
\begin{equation}
\label{eq:ut_cond}
\E[\int_0^T D_tg(\Xt) u_t] = 1
\end{equation}
\end{theorem}

\textbf{Proof outline}: We start with \eqref{eq:equiv_loss} and write $\delta(g(\Xt)$ as $\delta(G)$.
Then, by the Malliavin chain rule, the adjoint relation \eqref{eq:adjoint} and the Skorohod integrand condition \eqref{eq:ut_cond}, we have
\begin{align*}
&\E[\ell(\Xt)\,\delta(G)] \\&
=\E\!\left[\int_0^T (D_t(\ell(\Xt))\1_{\{g(\Xt)>0\}}))\,u_t\,dt\right]
\\&= \E\!\Big[\1_{\{g(\Xt)>0\}}\left(\ell(\Xt)\skor(u)-
\int_0^T (D_t\ell(\Xt))u_t\,dt\right)\Big].
\end{align*}
The denominator in \eqref{eq:equiv_loss} can be derived similarly. 

\paragraph*{Remarks}
(i) There is considerable flexibility  in the choice of $u$ in the above theorem.
The canonical choice is:
For $g(\Xt)  \in \mathbb{D}^{1,2}$ with Malliavin derivative 
$D g(\Xt) = \{ D_t g(\Xt) \}_{t \in [0,T]} \in L^2(\Omega; H)$, 
choose 
\begin{equation}
  \label{eq:generalu}
u_t = \frac{D_t g(\Xt)}{\|D g(\Xt)\|_H^2}, \qquad t \in [0,T]
\end{equation}
where $H$ is the Cameron--Martin space with norm
\[
\|h\|_H^2 \defn \int_0^T |h(t)|^2\,dt.
\]
The  choice~\eqref{eq:generalu} ensures that $u \in H$ and is always well-defined. 
However, in certain special cases one may use simpler (though less general) 
expressions. 
For example if $D_t g(\Xt) \neq 0$ a.e., one can choose
\begin{equation}
\label{eq:u_choice}
u_t = \begin{cases} \frac{1}{T D_t g(\Xt)}  & D_t g(\Xt) \neq 0 \\
     1 &  D_t g(\Xt) = 0 .
     \end{cases}
\end{equation}
But one has to be careful with the choice~\eqref{eq:u_choice}. For $g(\Xt) = \int_0^t W_s ds$,
then $D_t g(\Xt)  = T-t$ so that  $u_t = \frac{1}{T D_t g(\Xt)} = \frac{1}{T(T-t)}  \notin H$.
In comparison, choosing $u$ according to \eqref{eq:generalu} yields $u_t = 3(T-t)/T^3 \in H$.

(ii) 
The 
representation~\eqref{eq:malliavin} requires evaluation of Malliavin derivatives and Skorohod integrals, see \cite{GM05} for several examples. There are  several important consequences.  First, it restores the
$N^{-1/2}$ Monte--Carlo convergence rate even under singular conditioning, as
the event $\{g(X^\theta)=0\}$ no longer needs to be sampled directly.  Second,
the estimator admits substantial variance--reduction flexibility: the choice of
localizing function (indicator versus smooth approximation) and of admissible
weight process $u$ strongly influence efficiency, with optimal choices
characterizable via variational principles in Malliavin calculus.  Third, the
representation is compatible with standard discretizations of the forward SDE:
the Malliavin derivatives $D_t X^\theta$ admit recursive Euler--Maruyama
approximations, so one avoids additional kernel bandwidths or curse--of--dimensionality
issues inherent in regression--based methods.

\section{Weak Derivative Estimator}

Applying the quotient rule, 
it follows  from~\eqref{eq:malliavin} that
\begin{equation}
\label{eq:mall_grad}
\nabla_\theta \E[\ell(X^\theta)\mid g(\Xt) = 0]
= \frac{E_2\,\nabla_\theta E_1 - E_1\,\nabla_\theta E_2}{E_2^2}.
\end{equation}
In this section we construct a weak derivative based  algorithm  to estimate $\nabla_\th E_1$ and
$\nabla_\th E_2$ given the SDE~\eqref{eq:sde}. The resulting gradient estimate can then be fed into a stochastic gradient algorithm to minimize the loss $\Loss(\th)$. This weak-derivative method recovers a $O(1)$ variance scaling with respect to the  time horizon $T$, in contrast to score function methods which incur $O(T)$ variance scaling.


\subsection{Discrete-time Weak Derivative of Transition Probabilities}
We start with an Euler discretization of the sample path of the SDE~\eqref{eq:sde}. Let $\Sigma(x,t) := \sigma(x,t)\sigma(x,t)^\top$. The resulting discrete time process has the transition probability given by the multivariate Gaussian
\begin{align}
\begin{split}
\label{eq:gauskern}
&P_{\Delta t}^\theta(x,t,d x')=\mathcal N(x+\Delta t\, b_\theta(x,t),\,\Delta t\,\Sigma(x,t))
\end{split}
\end{align}
In order to analyze parameter sensitivities, one needs to differentiate the family of Markov transition probabilities
$\{P_{\Delta t}^\theta\}_{\theta}$ induced by this Euler discretization. Since each
$P_{\Delta t}^\theta$ is a probability measure on $\mathbb{R}^d$, its derivative with respect to $\theta$
is not a probability measure in general, but rather a \emph{signed measure}. More precisely, if
\[
\nabla_{\theta} P_{\Delta t}^\theta(x,t,\cdot)
\]
exists in the weak sense\footnote{To keep the notation simple and avoid multidimensional matrices, we assume $\theta$ is a scalar parameter. Dealing with $\th \in\reals^p$ simply amounts to interpreting the results elementwise.}, then it defines a bounded signed measure: for every smooth and bounded
test function $f$,
\begin{equation}
\label{eq:wd}
\nabla_{\theta} P_{\Delta t}^\theta f(x) \;=\; \nabla_{\theta}\int_{\reals^d} f(x') \, P_{\Delta t}^\theta(x,t,dx').
\end{equation}
\eqref{eq:wd} is called the weak-derivative\footnote{This weak-derivative is also called the  measure-valued derivative.} of $P_{\Delta t}^\theta$.
By the Hahn–Jordan decomposition theorem \cite{Bil86}, any signed measure $\nu$ on a measurable space can be
expressed as the difference of two mutually singular positive measures:
\[
\nu \;=\; \nu^+ - \nu^-,
\]
with $\nu^+, \nu^-$ uniquely determined. Applying this to the weak derivative
$\nabla_{\theta} P_{\Delta t}^\theta(x,t,\cdot)$, we obtain
\[
\nabla_{\theta} P_{\Delta t}^\theta \;=\; c_{\theta}(\rho_\theta^+ - \rho_\theta^-),
\]
where $\rho_\theta^\pm$ are positive normalized measures, and $c_{\theta}$ is a scaling factor weighting each measure equally. Specifically, for multivariate Gaussian transition probability $P_{\Delta t}^\theta$, the weak-derivative consists of the difference of two Weibull distributions, in each spatial dimension.

First, we formalize the existence of the  weak derivative of the transition probabilities.

\begin{theorem}[Discrete-Time Hahn--Jordan Weak Derivative]
\label{thm:hj_discrete}
Let $(X_t^\theta)_{t\in[0,T]} \subset \reals^n$ solve the Itô SDE \eqref{eq:sde} with $b_\theta \in C_b^2(\reals^n\times \reals;\reals^n)$, $\sigma \in C_b^2(\reals^n \times \reals;\reals^{n\times d})$, where $d$ is the dimension of Brownian motion.  

For $\Delta t > 0$, the Euler--Maruyama scheme induces the Gaussian transition probability \eqref{eq:gauskern}.
Then the weak derivative of $P^\theta_{\Delta t}$ with respect to~$\theta$ admits a Hahn--Jordan decomposition
\begin{equation}
\label{eq:hjd}
\nabla_{\theta} P^\theta_{\Delta t}(x,t,dx') 
= c_{\theta}(x,t)(\rho^+_\theta(x,t,dx') - \rho^-_\theta(x,t,dx')),
\end{equation}
where $\rho^\pm_\theta(x,t,\cdot)$ are mutually singular positive measures. 
Consequently, for any bounded measurable $f:\reals^d\to\reals$,
\begin{align*}
&\nabla_{\theta} \int f(x')\,P^\theta_{\Delta t}(x,t,dx')
\\&= c_{\theta}(x,t)\left(\int f(x')\,\rho^+_\theta(x,t,dx') - \int f(x')\,\rho^-_\theta(x,t,dx')\right).
\end{align*}
\end{theorem}

\begin{proof}
Since $P^\theta_{\Delta t}$ is Gaussian with mean 
$\mu_\theta(x,t) = x + \Delta t\,b_\theta(x,t)$ and covariance 
$\Sigma(x,t) = \Delta t\,a(x,t)$, the density is smooth in $\theta$ by 
the $C^2_b$ assumption\footnote{$C_b^2$ denotes twice-differentiable bounded functions, and $C_c^{\infty}$ denotes infinitely-differentiable functions with compact support.}. For any $\phi \in C_c^\infty(\reals^n)$,
\[
\nabla_{\theta} \int \phi(x')\,P^\theta_{\Delta t}(x,t,dx') 
= \int \phi(x')\, \nabla_{\theta} p_\theta(x,t,x')\,dx',
\]
where $p_\theta(x,\cdot)$ is the Gaussian density. Thus 
$\nabla_{\theta} P^\theta_{\Delta t}(x,t,\cdot)$ defines a finite signed measure. 
By the Hahn--Jordan decomposition theorem, every finite signed measure admits 
a unique decomposition into two mutually singular positive measures 
$\rho^+_\theta$ and $\rho^-_\theta$. Upon normalization, a common $c_{\theta}$ scale factor will be produced since
\begin{align*}
&\int\nabla_{\theta}P_{\Delta t}^{\theta}(x,t,x')dx' = \nabla_{\theta}\int P_{\Delta t}^{\theta}(x,t,x')dx' \\&= 0 = \int(c_{\theta}^+\rho_{\theta}^+(x') - c_{\theta}^-\rho_{\theta}^-(x'))dx' = c_{\theta}^+ - c_{\theta}^-
\end{align*}
\end{proof}

\begin{figure}
  \includegraphics[scale=0.3]{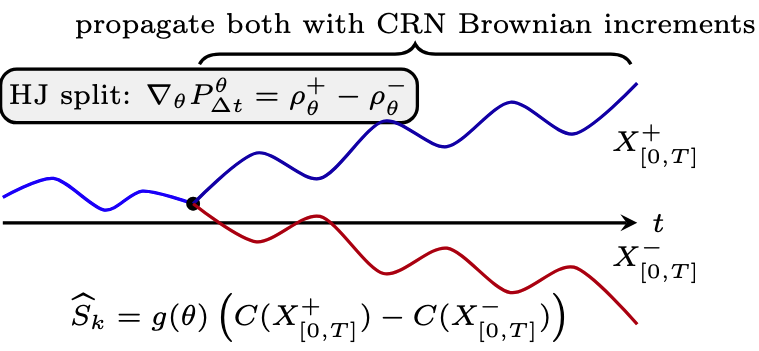}
  \caption{Conceptual Schematic of Hahn-Jordan Decomposition for Weak Derivative Estimator. CRN denotes common random number generation.}
\end{figure}

\textit{Algorithmic Motivation}: The above  weak-derivative representation is amenable to Monte Carlo
implementation: the signed derivative can be simulated by branching into two processes, one evolving
under $\rho_\theta^+$ with weight $+1$, the other under $\rho_\theta^-$ with weight $-1$.
Expectations against the signed measure can then be evaluated as weighted averages of functionals of
these branched processes. This is the weak-derivative simulation method in \cite{Pfl96}, which we outline below. 

\subsection{Weak-Derivative Estimator}
\label{sec:HJWD}

Here we aim to estimate a general gradient $\nabla_{\theta}\E[\ell(\Xt)]$, where the loss depends on the solution path $\Xt$ to the SDE \eqref{eq:sde}. We first outline our proposed weak-derivative estimation algorithm, and compare to the score function method. Then we relate this to computation of the Malliavin gradient \eqref{eq:mall_grad}. For simplicity of implementation, we assume that $\Xt_t$ starts at the stationary distribution of \eqref{eq:sde}. We refer to \cite{MT93} for conditions for exponentially ergodic diffusions with well-defined stationary distributions.

\subsubsection*{Weak-Derivative Gradient Estimation Algorithm} Let the discretization interval $\Delta t=T/M$.  The Euler transition probability at state $x$ is \eqref{eq:gauskern}, and its Hahn--Jordan decomposition is given by \eqref{eq:hjd}. Algorithm~\ref{alg:hjd} provides the methodology for Monte-Carlo estimation of the weak-derivative sensitivity via a Hahn-Jordan path-splitting technique.

\begin{algorithm}[H]\small
\caption{Single–branch HJ estimator (Euler transition probability, drift $b_\theta$, diffusion $\sigma$)}
\label{alg:hjd}
\begin{algorithmic}[1]
\Require continuous time horizon $T$, discrete time horizon $M$, $\Delta t=T/M$, parameter $\theta$, branch index $k$, functional $\wloss$
\State Simulate Euler path $X_0,\dots,X_k$.
\State Form $\rho_\theta^\pm(X_k,k\Delta t, \cdot)$ and $c_{\theta}^\pm(X_k,k\Delta t) =: g(\theta)$.
\State Draw $X_k^{(+)}\sim\rho_\theta^+(X_k,k\Delta t,\cdot)$, $X_k^{(-)}\sim\rho_\theta^-(X_k,k\Delta t,\cdot)$.
\State Generate future Gaussian increments $\{\xi_j\}_{j=k+1}^M$ and reuse them for both branches (CRN).
\State Propagate both branches by Euler from $t_k$ to $T$ using the same $\{\xi_j\}$, forming paths $X^{(+)}_{[0,T]}$ and $X^{(-)}_{[0,T]})$.
\State Return
\begin{equation}
\label{eq:s_k}
    \widehat S_k=g(\theta)\left(\wloss(X^{(+)}_{[0,T]})-\wloss(X^{(-)}_{[0,T]})\right)
\end{equation}
\end{algorithmic}
\end{algorithm}

We evaluate a weak-derivative estimator $\hat{S}_k$ for the gradient $\nabla_{\theta}\E[\ell(\Xt)]$ as \eqref{eq:s_k}
where we first branch according to the weak-derivative transition probability, then propagate the branched paths forward under nominal dynamics and common Gaussian increments, to form paths $X_{[0,T]}^+$ and $X_{[0,T]}^+$. This is exactly analogous to the weak-derivative algorithmic computation in \cite{Pfl96}.

This realizes the discrete weak derivative of a single Euler transition probability,
with \emph{order-1 variance} in $T$ due to a single local branch and synchronous
coupling thereafter. This is in contrast to the score function method \cite{Pfl96}, which incurs $O(T)$ variance scaling.

\textbf{Result} \cite{Pfl96}: Assume exponential ergodicity. Then 
\begin{equation*}
\underset{\Delta t \to 0}{\lim}\E[\widehat S_n]
= \nabla_\theta \E[\wloss(\Xt)], \,\,\,\textrm{Var}(\widehat S_n)=O(1) \textrm{ in } T
\end{equation*}

{\bf Remark}. In comparison, the score function estimate, widely used in reinforcement learning, has $O(T)$ variance growth:
 $$\hat{\nabla}\E[\wloss(\Xt)] \;=\; \frac{1}{N}\sum_{i=1}^N \wloss\!\big(X^{\theta,(i)}_{[0,T]}\big)\;\frac{\nabla_{\theta} p_\theta\!\big(X^{\theta,(i)}_{[0,T]}\big)}{p_\theta\!\big(X^{\theta,(i)}_{[0,T]}\big)}$$ but uses a single sample path.

\subsubsection*{Malliavin Gradient Estimation}
For the counterfactual stochastic gradient algorithm, 
recall from \eqref{eq:mall_grad}, that  we need to compute sensitivities $\nabla_{\theta} E_1^{\theta}$ and $\nabla_{\theta} E_2^{\theta}$, defined in \eqref{eq:malliavin}.  We  compute these sensitivities using Algorithm~\ref{alg:hjd} with the loss functional $\wloss$ replaced respectively by the loss functional of $E_1^{\theta}$ as 
\begin{equation*}
\ones_{g(\Xt)>0}\left(\ell(X^\th) \skor(u) -
  \int_0^T (D_t \ell(\Xt)) u_t dt \right)
\end{equation*}
and of $E_2^{\theta}$ as
\begin{equation*}
    \ones_{\{g(\Xt)>0\}} \skor(u)
\end{equation*}
where $u$ satisfies \eqref{eq:ut_cond}. Notice  that in Algorithm~\ref{alg:hjd}, we only need to plug in these loss functionals, and not their derivatives w.r.t. $\theta$, to compute $\nabla_{\theta} E_1^{\theta}$ and $\nabla_{\theta} E_2^{\theta}$. Recall that computation of $\skor(u)$ and $D_t\ell(\Xt)$ is attained as described in \eqref{eq:mall_form} and \eqref{eq:skor_exp} in Section~\ref{sec:mall_comp}; see~\cite{GM05}. In Section~\ref{sec:num} we illustrate such computation for an Ornstein-Uhlenbeck process.

\subsection{Connection to Infinitesimal Generator}

The weak-derivative estimator is traditionally applied to discrete-time processes. The aim of this section is to show that, by a limiting argument, this method applies to continuous-time diffusions.
Specifically, two complementary perspectives underlie the Hahn--Jordan weak derivative (HJ-WD) method. Discussed thus far is a discrete-time sample-path approach, which is infinitesimally equivalent to a distributional approach derived through the continuous-time Fokker-Planck generator.

\paragraph{Infinitesimal Generator Formulation}
Differentiating the Fokker--Planck equation with respect to~$\theta$ yields the
sensitivity PDE
\[
\partial_t \nu_t = (L^\theta)^{\!*}\nu_t + \nabla_{\theta} (L^\theta)^{\!*}p_t^\theta,
\]
where $\nu_t=\nabla_{\theta} p_t^\theta$. By the Duhamel formula,
\[
\nu_T = \int_0^T P_{T-s}^\theta \big( \nabla_{\theta} (L^\theta)^{\!*}p_s^\theta \big)\,ds,
\]
so the derivative measure at time~$T$ is an integral of \emph{signed mass injections}
$\nabla_{\theta} (L^\theta)^{\!*}p_s^\theta$ transported forward by the semigroup $P^\theta$.
A Hahn--Jordan decomposition can be applied to the signed measure $\nabla_{\theta} (L^\theta)^{\!*}p_s^\theta$.

\paragraph{Discrete Euler Formulation}
The Euler--Maruyama discretization induces transition probabilities \eqref{eq:gauskern}
\[
P^\theta_{\Delta t}(x,d x') = \mathcal{N}\!\big(x+\Delta t\,b_\theta(x),\,\Delta t\,\Sigma(x)\big).
\]
The weak derivative $\nabla_{\theta} P^\theta_{\Delta t}(x,\cdot)$ is itself a finite signed measure,
admitting a Hahn--Jordan decomposition
\(
\nabla_{\theta} P^\theta_{\Delta t} = \rho_\theta^+ - \rho_\theta^-.
\)
By simulating a \emph{single signed branch} at some time step and propagating both copies forward,
one obtains an unbiased estimator for the discretized weak derivative.
Randomization of the branch time recovers the full time integral in expectation, following the
measure-valued derivative framework for Markov chains.

\paragraph{Consistency}
The two perspectives are equivalent in the limit as $\Delta t \to 0$:
the generator-level source $\nabla_{\theta} (L^\theta)^{\!*}p_t^\theta$ is the infinitesimal analogue of the
Euler transition probability derivative $\nabla_{\theta} P^\theta_{\Delta t}$. The result in Section~\ref{sec:HJWD} assumed that we start in the stationary distribution; even if this is not so, we can still recover consistency of the weak-derivative estimator. 
As $\Delta t\to 0$, the Riemann-sum representation
\begin{align*}
    &\sum_{k=0}^{T/\Delta t-1} P^\theta_{T-(k+1)\Delta t}\,\nabla_{\theta} P^\theta_{\Delta t}\,P^\theta_{k\Delta t}
\;\\&\quad \underset{\Delta t \to 0}{\longrightarrow}\;
\int_0^T P^\theta_{T-s}\big(\nabla_{\theta} (L^\theta)^{\!*}p_s^\theta\big)\,ds
\end{align*}
justifies the equivalence\footnote{This limit is clear at least for the smooth Gaussian transition kernel \eqref{eq:gauskern}}.
Thus the Euler/HJ scheme is a Monte Carlo realization of the generator-level
HJ decomposition, with the same $O(1)$ variance properties but implementable in practice.

\section{Numerical Implementation. Ornstein-Uhlenbeck Process}
\label{sec:num}
Here we specify the SDE dynamics to an Ornstein-Uhlenbeck equation, and derive the necessary analytical expressions for the Malliavin numerator and denominator in \eqref{eq:malliavin}.
Despite the simplicity of this model, the evaluation of the conditional loss and its gradient are non-trivial due to the conditioning on a zero-probability event.

Assume we have $N$ simulated sample paths from the following diffusion \[
dX_t^\theta = -\theta X_t^\theta\, dt + \sigma\, dW_t,\quad X_0 = 0
\]
where $\theta > 0$ lies in some compact set $\Th \subset \reals$. The aim is to estimate the counterfactual conditional loss 
\[\nabla_{\theta}\CE[X_1^2 | X_{0.5} = 0].  \] 

Using \eqref{eq:mall_form}, the Malliavin derivative for the Ornstein–Uhlenbeck process is given explicitly by:
\begin{equation}
D_s X_t = \sigma e^{-\theta(t - s)} \ind_{{0 \leq s \leq t}}.
\end{equation}

Therefore, the Malliavin derivative of $g(\Xt) = X_{0.5}$ is:
\begin{equation}
\label{eq:mall_g}
D_s g(\Xt) = D_s X_{0.5} = \begin{cases}
0, & X_t = 0 \\
\sigma e^{-\theta(t - s)}\ind_{{0 \leq s \leq 0.5}}, & X_t \neq 0
\end{cases}
\end{equation}

Therefore the Skorohod integrand process $u$ in~\eqref{eq:u_choice} is
\begin{equation*}
u_s = \begin{cases}
\frac{1}{T D_s g(\Xt)}, & D_s g(\Xt) \neq 0, \\
1, & D_s g(\Xt) = 0.
\end{cases}
\end{equation*}
Then  the conditional expectation can be represented as~\eqref{eq:malliavin}:
\begin{align}
\begin{split}
\label{eq:mall_ce}
&\mathbb{E}[X_1^2 \mid X_{0.5} = 0] \\&= \frac{\mathbb{E}\left[X_1^2 \ind_{{X_{0.5} > 0}} \skor(u) - \ind_{{X_{0.5} > 0}} \int_0^1 (D_s X_1^2) u_s ds\right]}{\mathbb{E}[\ind_{{X_{0.5} > 0}}\skor(u)]}. 
\end{split}
\end{align}

In order to compute \eqref{eq:mall_ce}, we need to compute two quantities: $\skor(u)$ and $D_sX_1^2$. 
\begin{itemize}
\item $\skor(u)$: From~\eqref{eq:mall_g},  $D_s g(\Xt)$ is deterministic and is thus trivially adapted. Recall, when $u$ is adapted to the filtration $\{\CF_t\}_{t\geq0}$, the Skorohod integral $\skor(u)$ coincides with the It\`o integral $\int_0^T u_t\, dW_t$. So $\skor(u)$ is  the It\`o integral
\[\skor(u) = \int_0^{0.5} \frac{1}{\sigma e^{-\theta(0.5-s)}} dW_s+  \int_{0.5}^T dW_s\]

\item $D_sX_1^2$: By the Malliavin chain rule we have is:
\begin{equation}
D_s X_1^2 = 2X_1 D_s X_1 = 2 X_1 \sigma e^{-\theta(1 - s)}\ind_{{0 \leq s \leq 1}}.
\end{equation}
\end{itemize}

Thus, the final explicit Malliavin calculus formulation for the Ornstein–Uhlenbeck conditional loss is 
\begin{align*}
&\mathbb{E}[X_1^2 \mid X_{0.5} = 0] \\&= \biggl\{-\mathbb{E}\biggl[X_1^2 \ind_{{X_{0.5} > 0}} \left(\int_0^{0.5} \frac{1}{\sigma e^{-\theta(0.5-s)}} dW_s+ \int_{0.5}^1 dW_s\right) \\&\qquad - \ind_{X_{0.5} > 0} \int_0^{0.5} (2X_1\sigma e^{-\theta(1-s)} \frac{1}{\sigma e^{-\theta(0.5 - s)}} ds)\biggr]\biggr\}\\&\qquad\times \biggl\{\mathbb{E}[\ind_{{X_{0.5} > 0}} \left(\int_0^{0.5} \frac{1}{\sigma e^{-\theta(0.5-s)}} dW_s + \int_{0.5}^T dW_s\right)\biggr\}^{-1}.
\end{align*}
Thus, we may compute sensitivity by the quotient rule \eqref{eq:mall_ce}, with
\begin{align}
\begin{split}
\label{eq:E_forms}
    E_1 &= \mathbb{E}\biggl[X_1^2 \ind_{{X_{0.5} > 0}} \left(\int_0^{0.5} \frac{1}{\sigma e^{-\theta(0.5-s)}} dW_s+ \int_{0.5}^1 dW_s\right) \\&\qquad\quad- \ind_{X_{0.5} > 0} \int_0^{0.5} (2X_1\sigma e^{-\theta(1-s)} \frac{1}{\sigma e^{-\theta(0.5 - s)}} ds)\biggr] \\
    & = \mathbb{E}\biggl[X_1^2 \ind_{{X_{0.5} > 0}} \left(\frac{1}{\sigma}\int_0^{0.5} e^{\theta(0.5-s)} dW_s+ \int_{0.5}^1 dW_s\right) \\&\qquad\quad- \ind_{X_{0.5} > 0} X_1  e^{-0.5\theta} \biggr] \\
    E_2 &= \mathbb{E}\left[\ind_{{X_{0.5} > 0}} \left(\frac{1}{\sigma}\int_0^{0.5} e^{\theta(0.5-s)} dW_s + \int_{0.5}^1 dW_s\right)\right]
\end{split}
\end{align}

$E_1$ and $E_2$ can be evaluated numerically by Monte-Carlo simulations (taking into account the event $\ones_{X_{0.5}>0}$) from sample paths, and the gradients $\nabla_{\theta}E_1, \nabla_{\theta}E_2$ are computed by Algorithm~\ref{alg:hjd}. 

We now verify that this approach incurs substantial computational advantage via the two complexity features:
\begin{enumerate}
    \item We observe standard $O(N^{-1/2})$ Monte-Carlo error convergence of the estimator, illustrated in Figure~\ref{fig:mc_convg}. This is in stark contrast to direct conditional Monte-Carlo estimators, which are infeasible in this case due to the \textit{measure-zero} conditioning event. 
    \item We observe $O(1)$ variance scaling with respect to the time horizon $T$. This is in contrast to the standard $O(T)$ variance scaling incurred by score function estimators. This disparity is illustrated in Figure~\ref{fig:varscale}
\end{enumerate}

All the results in this paper are fully reproducible. The code that generated the numerical results and figures can be downloaded from  \texttt{\small https://github.com/LukeSnow0/Malliavin-WD}.
\begin{figure}
\centering
\includegraphics[width=0.8\linewidth]{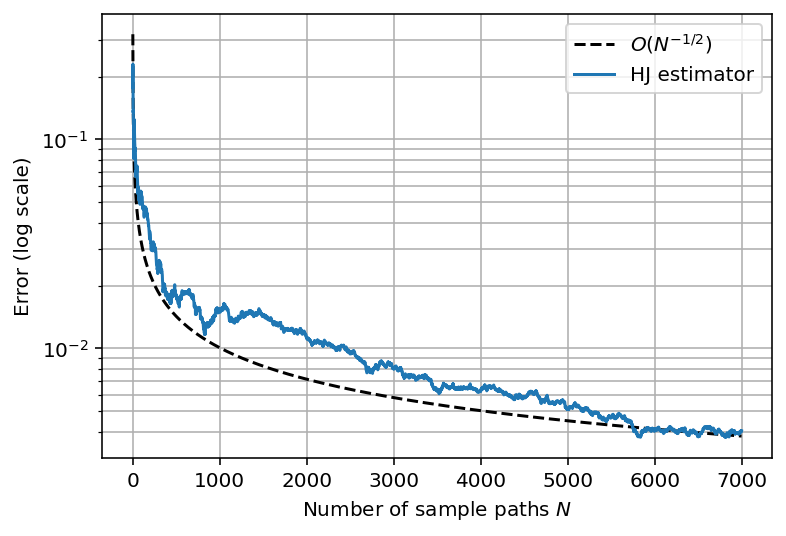}
\caption{Convergence of Ornstein-Uhlenbeck Malliavin quotient \eqref{eq:mall_ce}, with respect to simulated paths $N$. We see that \eqref{eq:mall_ce} recovers a $O(N^{-1/2})$ convergence rate \textit{even though we condition on a measure-zero event}. }\label{fig:mc_convg}
\end{figure}

\begin{figure}[t]
    \centering
    \includegraphics[width=0.8\linewidth]{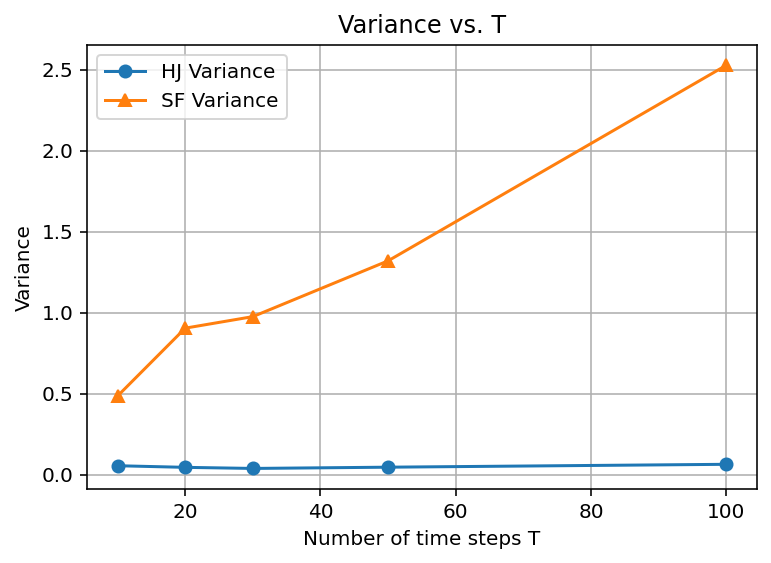}
    \caption{Variance scaling of the weak derivative estimator and the score function estimator, for varying time horizon $T$. We verify the stable $O(1)$ variance scaling of the weak derivative estimator, in contrast to the $O(T)$ variance scaling of the score function estimator.}\label{fig:varscale}
    \label{fig:varscale}
\end{figure}

\section{Conclusion}

We have presented a methodology for counterfactual stochastic optimization of conditional loss functionals. As explained, this procedure can be viewed as a form of passive learning. Instead of relying on kernel methods  or direct conditional Monte-Carlo, we exploit a reformulation of the conditional expectation by Malliavin calculus. This allows for recovery of the $O(N^{-1/2})$ Monte-Carlo convergence rate even when conditioning on rare or measure-zero events, where direct Monte-Carlo becomes impossible and kernel smoothing methods infeasible and inefficient. Furthermore, we combine this approach with a weak-derivative gradient estimation algorithm which incurs stable $O(1)$ variance scaling in the time-horizon, in contrast to score function methods which scale as $O(T)$. In future work it is worthwhile generalizing the above approach to counterfactual Langevin dynamics type algorithms.

Finally, recall  that classical counterfactual risk evaluation seeks to evaluate $\E_{p(x|\beta)}\{ \loss(x)\}$ given simulations of $\loss(x)$ drawn from $p(x|\alpha)$. 
Then clearly $$\E_{p(x|\beta)}\{ \loss(x) \} = \E_{p(x|\alpha)} \{ \loss(X) p(X|\beta)/p(X|\alpha) \}. $$
In this paper, we extend this setting to continuous time,  where both the loss and conditioning event are  functionals of the trajectory. Then expressions for $p(X|\beta)$ and $p(X|\alpha)$ are not available in closed form.  The classical  importance-sampling identity relies on  absolute continuity of $p(\cdot|\alpha)$ and $p(\cdot|\beta)$. 
In our continuous-time framework, however, the conditioning event is a zero-probability path
functional ($g(\Xt)=0$), 
so the ratio $p(X|\beta)/p(X|\alpha)$ is no longer meaningful. 
To address this, we exploit the known dynamics of the SDE~\eqref{eq:sde} and  replace the likelihood ratio by a Malliavin calculus representation involving Dirac delta functionals and 
Skorohod integrals, which yields a constructive estimator of the conditional expectation.

\bibliographystyle{IEEEtran}
\bibliography{$HOME/texstuff/styles/bib/vkm}

\end{document}